\documentclass[11pt,reqno]{amsart}
\usepackage[T1]{fontenc}
\usepackage{lmodern}
\usepackage{amsmath,amssymb,amsthm,mathtools}
\usepackage[margin=1.12in]{geometry}
\usepackage[expansion=false]{microtype}
\usepackage{enumitem}
\usepackage{booktabs}
\usepackage{xcolor}
\usepackage[hypertexnames=false,colorlinks=true,linkcolor=blue!45!black,citecolor=blue!45!black,urlcolor=blue!45!black]{hyperref}
\numberwithin{equation}{section}
\newtheorem{theorem}{Theorem}[section]

\newtheorem{lemma}[theorem]{Lemma}
\newtheorem{corollary}[theorem]{Corollary}
\theoremstyle{definition}
\newtheorem{definition}[theorem]{Definition}
\newtheorem{example}[theorem]{Example}
\theoremstyle{remark}

\newcommand{\R}{\mathbb R}
\newcommand{\Z}{\mathbb Z}
\newcommand{\HH}{\mathcal H}
\newcommand{\Reg}{\mathcal R}
\newcommand{\Sing}{\mathcal S}
\newcommand{\Man}{\mathcal U}
\newcommand{\Ric}{\operatorname{Ric}}
\newcommand{\Scal}{\operatorname{Scal}}
\newcommand{\Vol}{\operatorname{Vol}}
\newcommand{\diam}{\operatorname{diam}}
\newcommand{\dist}{\operatorname{dist}}
\newcommand{\dimH}{\dim_{\mathrm H}}
\newcommand{\GH}{\mathrm{GH}}
\newcommand{\Tube}{\mathcal T}

\setlist[enumerate]{label=\textup{(\roman*)},leftmargin=2.1em}
\hypersetup{pdftitle={Codimension-four regularity of noncollapsed Ricci limit spaces under an integral volume-deficit bound},pdfauthor={Lingling Kong}}

\title[Volume deficits and codimension-four regularity]{Codimension-four regularity of noncollapsed Ricci limit spaces under an integral volume-deficit bound}
\author{Lingling Kong}
\address[Lingling Kong]{School of Mathematics and Statistics, Northeast Normal University, Changchun, China}
\email{kongll111@nenu.edu.cn}

\date{\today}
\thanks{The author thanks Professor Xiaochun Rong for his valuable comments and suggestions.
The author is supported by the National Natural Science Foundation of China under Grant No.~12271372.
AI assistance was used in formulating condition \textup{(AC)} based on \cite{Kong2026-1}.}
\begin{document}
\begin{abstract}
Let $(X,d,p)$ be an $n$-dimensional noncollapsed Ricci limit space,
where $n\ge4$. Under an integral volume-deficit bound, we prove that
the metric singular set has Hausdorff codimension at least four and
sigma-finite $(n-4)$-dimensional Hausdorff measure. This establishes
a special case of the codimension-four regularity conjecture.
Moreover, the nonmanifold locus is closed and has locally finite
$(n-4)$-dimensional Hausdorff measure, and its intersection with each
bounded ball satisfies a tubular-volume estimate of order $r^4$.
In dimension four, the nonmanifold points form a locally finite set.
A flat quotient example shows that the codimension bound is sharp.
\end{abstract}
\maketitle

\section{Introduction and main theorem}

We consider noncollapsed Ricci limit spaces arising as pointed
Gromov--Hausdorff limits
\begin{equation}\label{eq:setting}
 (M_i^n,g_i,p_i)\xrightarrow{\GH}(X,d,p),\qquad
 \Ric_{g_i}\ge -(n-1)g_i,\qquad
 \Vol_{g_i}(B_1(p_i))\ge v>0,
\end{equation}
where the manifolds are connected, complete and without boundary. We
write $\mu=\HH^n$, with Hausdorff measure normalized to agree with
Riemannian volume, and let $\omega_n$ denote the volume of the Euclidean
unit ball. Volume convergence identifies $\mu$ with the limit of the
Riemannian volume measures \cite{Colding1997,CC1997I}. All balls are
open unless otherwise stated.

The metric regular set $\Reg$ consists of points whose tangent cones
are all isometric to $\R^n$. By Cheeger--Colding theory, the metric
singular set $\Sing=X\setminus\Reg$ has Hausdorff dimension at most
$n-2$ \cite{CC1997I}. Points outside $\Reg$ may nevertheless admit
manifold neighborhoods. For example, a metric cone may be homeomorphic
to Euclidean space even when its vertex is metrically singular. Conjecturally, every
noncollapsed Ricci limit space is a topological manifold outside a
closed set of Hausdorff codimension at least four; see
\cite[Conjecture~0.7]{CC1997I} and the discussion in
\cite{BruePigatiSemola2024+}. Under two-sided Ricci curvature bounds,
codimension-four regularity was established by Cheeger--Naber
\cite{CheegerNaber2015}.

In this paper, we obtain codimension-four regularity under an integral
bound on the volume deficits of small balls in the limit space.
The estimates apply to the metric singular set. A fixed density
threshold for manifold neighborhoods gives locally finite Hausdorff
measure and corresponding local estimates for the nonmanifold locus.

Define the hyperbolic comparison volume and the volume deficit by
\begin{equation}\label{eq:VandD}
 V(r)=n\omega_n\int_0^r(\sinh t)^{n-1}\,dt,
 \qquad
 D_r(x)=1-\frac{\mu(B_r(x))}{V(r)}.
\end{equation}
Bishop--Gromov comparison gives $0\le D_r(x)\le1$, and
$r\mapsto D_r(x)$ is nondecreasing.

\begin{definition}[Quadratic integral volume-deficit condition]\label{def:AC}
We say that $X$ satisfies \textup{(AC)} if, for every $R>0$, there
are constants $C_R<\infty$ and $r_R>0$ such that
\begin{equation}\tag{AC}\label{eq:AC}
 \int_{B_R(p)}D_r(x)^2\,d\mu(x)\le C_Rr^4
 \qquad(0<r<r_R).
\end{equation}
The constants may depend on $X$ and $R$, but not on $r$.
\end{definition}

Define the volume density and the density-deficit sets by
\begin{equation}\label{eq:density}
 \Theta(x)=\lim_{r\rightarrow0}\frac{\mu(B_r(x))}{V(r)},
 \qquad
 S_\delta=\{x\in X:\Theta(x)\le1-\delta\},
 \quad 0<\delta<1.
\end{equation}
For $F\subset X$ and $\rho>0$, write
\[
 \Tube_\rho(F)=\{y\in X:\dist(y,F)<\rho\}.
\]
Let $N_\rho(F)$ be the least number of open balls of radius $\rho$,
with centers in $X$, needed to cover $F$.

\begin{theorem}\label{thm:main}
Suppose that \eqref{eq:setting} holds, $n\ge4$, and $X$ satisfies
\textup{(AC)}. Then the following hold.
\begin{enumerate}
\item For every $R>0$ and $0<\delta<1$, there are constants
$K_{R,\delta}<\infty$ and $\rho_{R,\delta}>0$ such that
\begin{align}
 N_\rho(S_\delta\cap B_R(p))
 &\le K_{R,\delta}\rho^{4-n},\label{eq:main-cover}\\
 \mu\bigl(\Tube_\rho(S_\delta\cap B_R(p))\bigr)
 &\le K_{R,\delta}\rho^4\label{eq:main-tube}
\end{align}
for $0<\rho<\rho_{R,\delta}$. In particular,
\begin{equation}\label{eq:main-finite}
 \HH^{n-4}(S_\delta\cap B_R(p))<\infty.
\end{equation}
\item The metric singular set satisfies
\begin{equation}\label{eq:main-dim}
 \dimH\Sing\le n-4,
\end{equation}
and $\HH^{n-4}$ is sigma-finite on $\Sing$.
\item Define the topological manifold locus by
\begin{equation}\label{eq:manifold-locus}
 \Man=\{x\in X:\text{$x$ has an open neighborhood homeomorphic to $\R^n$}\}.
\end{equation}
Then $E=X\setminus\Man$ is closed, and $\Man$ is an open dense
topological $n$-manifold without boundary. For every $R>0$,
\begin{equation}\label{eq:top-finite}
 \HH^{n-4}(E\cap B_R(p))<\infty.
\end{equation}
Moreover, for some $K_R<\infty$ and all sufficiently small $\rho>0$,
\begin{align}
 N_\rho(E\cap B_R(p))&\le K_R\rho^{4-n},\label{eq:top-cover}\\
 \mu\bigl(\Tube_\rho(E\cap B_R(p))\bigr)&\le K_R\rho^4.
 \label{eq:top-tube}
\end{align}
Consequently, $\dimH E\le n-4$.
\end{enumerate}
\end{theorem}

\begin{corollary}\label{cor:four}
If $n=4$, the metric singular set is at most countable, whereas the
nonmanifold points form a locally finite set. Thus $X$ is a topological
$4$-manifold outside a closed set that meets each bounded ball in
finitely many points.
\end{corollary}

Condition \textup{(AC)} is an additional hypothesis. It holds on
bounded sets in every fixed smooth complete manifold satisfying the
Ricci lower bound. However, it may fail in a noncollapsed limit if
the constants are not controlled uniformly along the sequence.
A cone example exhibiting this failure is given in
\cite[Example 4.1]{Kong2026-1}. Section~\ref{sec:examples} gives
a flat quotient example for which the codimension-four bound is attained.

\section{Volume comparison and density deficits}\label{sec:prelim}

We recall the consequences of Ricci limit theory used below.
The space $X$ is proper and separable, and $\mu$ has full support
and is finite on bounded sets. Bishop--Gromov comparison gives
\begin{equation}\label{eq:BG}
 \frac{\mu(B_s(x))}{V(s)}
 \ge \frac{\mu(B_t(x))}{V(t)}
 \qquad(0<s\le t).
\end{equation}
Noncollapsed volume comparison also gives $\mu(B_t(x))\le V(t)$.
For this upper volume bound and the associated density rigidity,
see \cite{CC1997I,DePhilippisGigli2018, Kong2026-1}.

\begin{lemma}[Local lower volume bound]\label{lem:lower}
For every $L>0$, there is $c_L>0$ such that
\begin{equation}\label{eq:lower}
 \mu(B_t(z))\ge c_Lt^n
 \qquad(z\in B_L(p),\ 0<t\le1).
\end{equation}
Moreover, there is a dimensional constant $b_n<\infty$ such that
\begin{equation}\label{eq:upper}
 \mu(B_t(z))\le V(t)\le b_nt^n\qquad(0<t\le1).
\end{equation}
\end{lemma}

\begin{lemma}[Density and the regular set]\label{lem:density}
The density in \eqref{eq:density} exists at every point and satisfies
\begin{equation}\label{eq:theta-properties}
 0<\Theta(x)\le1,\qquad
 \Theta(x)=\lim_{r\rightarrow0}\frac{\mu(B_r(x))}{\omega_nr^n},
 \qquad
 \Reg=\{x:\Theta(x)=1\}.
\end{equation}
Furthermore, $\Theta$ is lower semicontinuous, each $S_\delta$ is
closed, and
\begin{equation}\label{eq:sing-union}
 \Sing=\bigcup_{m=2}^{\infty}S_{1/m}.
\end{equation}
For $x\in S_\delta$, one has
\begin{equation}\label{eq:persistent}
 \mu(B_t(x))\le(1-\delta)V(t)\qquad(t>0).
\end{equation}
\end{lemma}

\begin{lemma}[Propagation of density deficits to nearby centers]\label{lem:propagation}
For every $0<\delta<1$, there is
$\eta=\eta(n,\delta)\in(0,1/4)$ such that, for $x\in S_\delta$
and $0<r\le1$,
\begin{equation}\label{eq:propagation}
 D_r(y)\ge\frac{\delta}{2}
 \qquad(y\in B_{\eta r}(x)).
\end{equation}
\end{lemma}

\begin{lemma}[A density threshold for manifold neighborhoods]\label{lem:topological}
There exists $\delta_*=\delta_*(n)\in(0,1)$ such that
\begin{equation}\label{eq:density-threshold}
 \Theta(x)>1-\delta_*
 \quad\Longrightarrow\quad x\in\Man.
\end{equation}
In particular, $E\subset S_{\delta_*}$.
\end{lemma}
\begin{proof}[Proof sketch]
Volume almost-rigidity gives Euclidean approximations in normalized
Gromov--Hausdorff distance. To apply the local metric Reifenberg
theorem, these approximations must hold uniformly at nearby centers
and at all sufficiently small scales. The theorem then yields a
manifold neighborhood of $x$.
\end{proof}

\section{Proof of the main theorem}\label{sec:proof}

Fix $R>0$ and $0<\delta<1$, and put
$F=S_\delta\cap B_R(p)$. Let $\eta$ be as in
Lemma~\ref{lem:propagation}. All radii are taken sufficiently small
for \textup{(AC)} to apply on $B_{R+1}(p)$.

\subsection{Packing and Hausdorff measure}
For $0<r<\min\{1,r_{R+1}\}$, choose a maximal family of pairwise
disjoint balls
\[
 \{B_{\eta r}(x_j)\}_{j=1}^{N},\qquad x_j\in F.
\]
The family is finite: each ball has volume at least
$c_{R+1}(\eta r)^n$, and all the balls lie in $B_{R+1}(p)$,
which has finite measure. By maximality,
\begin{equation}\label{eq:cover}
 F\subset\bigcup_{j=1}^{N}B_{2\eta r}(x_j).
\end{equation}
Indeed, the ball of radius $\eta r$ centered at any uncovered
point could be added to the family.

By \textup{(AC)} and Lemmas~\ref{lem:lower}
and~\ref{lem:propagation},
\begin{align}
 C_{R+1}r^4
 &\ge\int_{B_{R+1}(p)}D_r(y)^2\,d\mu(y)\notag\\
 &\ge\sum_{j=1}^{N}\int_{B_{\eta r}(x_j)}D_r(y)^2\,d\mu(y)\notag\\
 &\ge N(\delta/2)^2c_{R+1}\eta^nr^n.
 \label{eq:packing-calculation}
\end{align}
Consequently,
\begin{equation}\label{eq:packing}
 N\le A_{R,\delta}r^{4-n},\qquad
 A_{R,\delta}=
 \frac{C_{R+1}}{(\delta/2)^2c_{R+1}\eta^n}.
\end{equation}
Taking $\rho=2\eta r$ in \eqref{eq:cover}--\eqref{eq:packing}
gives \eqref{eq:main-cover}.

For $n>4$, the same covering satisfies
\[
 \sum_{j=1}^{N}
 \bigl(\diam B_{2\eta r}(x_j)\bigr)^{n-4}
 \le A_{R,\delta}(4\eta)^{n-4}.
\]
Letting $r\rightarrow0$ proves \eqref{eq:main-finite}, with the
normalization constant for Hausdorff measure absorbed into the bound. When $n=4$, \eqref{eq:packing} bounds the
number of covering balls independently of $r$, which forces $F$
to be finite. Indeed, any finite subset of $F$ whose cardinality exceeds this bound
has positive minimum pairwise distance and cannot be covered by
so few sufficiently small balls.
This proves \eqref{eq:main-finite} in the case $n=4$ as well.

For every $q>n-4$, the same covering gives
\[
 \sum_{j=1}^{N}
 \bigl(\diam B_{2\eta r}(x_j)\bigr)^q
 \le A_{R,\delta}(4\eta)^q r^{q+4-n}
 \longrightarrow0.
\]
It follows that $\dimH F\le n-4$. By Lemma~\ref{lem:density},
\[
 \Sing=
 \bigcup_{m=2}^{\infty}\ \bigcup_{\ell=1}^{\infty}
 \bigl(S_{1/m}\cap B_\ell(p)\bigr).
\]
Countable stability of Hausdorff dimension gives
\eqref{eq:main-dim}. Each set in this union has finite $\HH^{n-4}$-measure by the
preceding argument. Hence $\HH^{n-4}$ is sigma-finite on $\Sing$.

\subsection{Tubular volume}
Set $r=\rho/\eta$ and choose $\rho$ sufficiently small that
$r<\min\{1,r_{R+1}\}$ and $\rho<1$. Then
$\Tube_\rho(F)\subset B_{R+1}(p)$. Every point of this neighborhood lies in $B_{\eta r}(x)$ for
some $x\in S_\delta$. Thus
Lemma~\ref{lem:propagation} and \textup{(AC)} give
\begin{align*}
 (\delta/2)^2\mu(\Tube_\rho(F))
 &\le\int_{\Tube_\rho(F)}D_{\rho/\eta}(y)^2\,d\mu(y)\\
 &\le C_{R+1}(\rho/\eta)^4.
\end{align*}
This proves \eqref{eq:main-tube}.

\subsection{The closed topological exceptional set}
The manifold locus $\Man$ is open: every point in a Euclidean chart
has a smaller neighborhood homeomorphic to Euclidean space. Thus
$E=X\setminus\Man$ is closed. By Lemma~\ref{lem:topological},
\begin{equation}\label{eq:E-fixed-density}
 E\subset S_{\delta_*}\subset\Sing.
\end{equation}
Applying part \textup{(i)} with $\delta=\delta_*$ and using
monotonicity under inclusion gives
\eqref{eq:top-finite}--\eqref{eq:top-tube}.

Since $\Man$ is an open subset of the separable metric space $X$,
it is Hausdorff and second countable. Its local Euclidean charts
therefore make it a topological $n$-manifold without boundary. Since
$\dimH E\le n-4<n$ and $\mu$ has full support, $E$ has empty
interior. Hence $\Man$ is dense. This completes the proof of
Theorem~\ref{thm:main}.

When $n=4$, each set $S_{1/m}\cap B_\ell(p)$ is finite, and hence
their countable union $\Sing$ is at most countable. Moreover,
\eqref{eq:top-finite} gives $\HH^0(E\cap B_R(p))<\infty$,
which proves the local finiteness asserted in
Corollary~\ref{cor:four}.\qed

\section{Geometric interpretation and examples}\label{sec:examples}

\subsection{Smooth small-ball asymptotics}
On a smooth Riemannian manifold, the volume of a small geodesic ball
admits the expansion
\begin{equation}\label{eq:smooth-volume}
 \Vol(B_r(x))=\omega_nr^n
 \left(1-\frac{\Scal(x)}{6(n+2)}r^2+O(r^4)\right);
\end{equation}
here the remainder is uniform as the center varies over a fixed compact
set \cite{GrayVanhecke1979}. Since the comparison space has sectional
curvature $-1$ and scalar curvature $-n(n-1)$,
\[
 V(r)=\omega_nr^n
 \left(1+\frac{n(n-1)}{6(n+2)}r^2+O(r^4)\right).
\]
Consequently,
\begin{equation}\label{eq:smooth-deficit}
 D_r(x)=\frac{\Scal(x)+n(n-1)}{6(n+2)}r^2+O(r^4).
\end{equation}
For every compact measurable set $K$ in a smooth complete manifold,
this uniform expansion gives
\begin{equation}\label{eq:scalar-limit}
 \lim_{r\rightarrow0}r^{-4}\int_K D_r(x)^2\,d\Vol(x)
 =
 \frac{1}{36(n+2)^2}
 \int_K\bigl(\Scal(x)+n(n-1)\bigr)^2\,d\Vol(x).
\end{equation}
This identity explains the scaling in \textup{(AC)}. In particular,
every fixed smooth complete manifold with the prescribed Ricci lower
bound satisfies \textup{(AC)} on bounded sets, with constants depending
on the local geometry. On a nonsmooth limit, \textup{(AC)} is used
only as a condition on metric ball volumes; no interpretation in terms
of $L^2$ scalar curvature is needed.

\subsection{Sharpness of codimension four}

\begin{example}[A flat quotient with a codimension-four nonmanifold set]\label{ex:sharp}
For $n\ge4$, let
\[
 X_n=\R^{n-4}\times(\R^4/\{\pm1\}),\qquad
 S=\R^{n-4}\times\{[0]\},
\]
with the product quotient metric and basepoint $p=(0,[0])$.
The Eguchi--Hanson metric is a complete Ricci-flat metric with
asymptotic cone $\R^4/\{\pm1\}$
\cite{EguchiHanson1979}; see also
\cite[Remark~1.2]{BruePigatiSemola2024+}.
Taking the product of each member of its blow-down sequence with
$\R^{n-4}$ realizes $X_n$ as a noncollapsed limit of smooth
Ricci-flat $n$-manifolds.

To verify \textup{(AC)}, retain the hyperbolic comparison volume
and put
\[
 h_n(r)=1-\frac{\omega_nr^n}{V(r)}=O_n(r^2).
\]
For $x=(z,[u])$ with $|u|>r$, the inverse image of $B_r(x)$
under the quotient map consists of two disjoint Euclidean balls
of radius $r$. Consequently,
\[
 \mu(B_r(x))=\omega_nr^n,\qquad D_r(x)=h_n(r).
\]
On the remaining region, we use $D_r\le1$ to obtain
\begin{align*}
 \int_{B_R(p)}D_r(x)^2\,d\mu(x)
 &\le\mu\bigl(\{|u|\le r\}\cap B_R(p)\bigr)
      +h_n(r)^2\mu(B_R(p))\\
 &\le C_{n,R}r^4.
\end{align*}
The last inequality follows since the transverse quotient ball
has volume $\frac12\omega_4r^4$ and the Euclidean factor is
restricted to a bounded ball. Hence \textup{(AC)} holds.

The complement of $S$ is a smooth manifold. For $x\in S$, local
homology gives
\[
 H_{n-2}(X_n,X_n\setminus\{x\};\Z)
 \cong\widetilde H_1(\mathbb{RP}^3;\Z)
 \cong\Z/2\Z.
\]
Local homology at the vertex of a cone is the reduced homology of
its link shifted by one degree; taking a product with $\R^{n-4}$
introduces a further shift by $n-4$. At a point of a topological $n$-manifold, the local homology
group in this degree vanishes. Thus every point of $S$ is a nonmanifold point, and
$E=S$. Since $\dimH S=n-4$, the codimension bound in
Theorem~\ref{thm:main} is sharp under \textup{(AC)}.
\end{example}

\bibliographystyle{plain-no-oxford}
\bibliography{reference}
\end{document}